\documentclass[12pt]{article}

\usepackage{amssymb}

\usepackage{common}

\title{On $\lam$-existence  over a predicate}

\author{Alexander Usvyatsov\thanks{ 
The author thanks the Austrian Science Foundation (FWF), projects P33895 and  P33420, for supporting this research. 
}}
 
\newcommand{\Addresses}{{
  \bigskip
  \footnotesize

Alexander Usvyatsov, \textsc{Institut f\"{u}r Diskrete Mathematik und Geometrie,
TU Wien,
1040, Vienna, Austria}

}} 
 
\newtheorem{theorem}{Theorem}[section]

\newtheorem{definition}[theorem]{Definition}
\newtheorem{example}[theorem]{Example}
\newtheorem{lem}[theorem]{Lemma}
\newtheorem{obs}[theorem]{Observation}
\newtheorem{co}[theorem]{Corollary}

\newtheorem{hyp}[theorem]{Hypothesis}
\newtheorem{remark}[theorem]{Remark}

\newtheorem{conv}[theorem]{Convention}

\newtheorem{ft}[theorem]{Fact}
\renewenvironment{proof}{\noindent {\em Proof:}}{\hspace*{1cm}
        \hspace*{\fill}$\rule{1.2ex}{1.4ex}$\medskip} 
\newenvironment{re}{\begin{remark}\rm}{\end{remark}}
 
\newenvironment{de}{\begin{definition}\rm}{\end{definition}}
\newenvironment{fact}{\begin{ft}\rm}{\end{ft}}

\date{}

\begin{document}

%

\maketitle

\abstract{We prove that in a countable theory $T$ fully stable over a predicate $P$, any $\lam$-complete set $A$ has the $\lam$-existence property. This means that $A$ can be extended to a $\lam$-saturated model of $T$ without changing the $P$-part.  The notion of $\lam$-completeness, introduced in this paper, captures some obvious necessary conditions for such an extension to be possible (for example, the $P$-part of $A$ has to be a $\lam$-saturated model of the appropriate theory). 
So in a fully stable theory $T$, $\lam$-existence can only fail for trivial reasons. This generalizes results of Chatzidakis in the context of difference fields of characteristic 0.

\section{Introduction}

Let $T$ be a complete first order theory, $P$ a distinguished unary predicate. Let $T^P$ be the complete first order theory of $P$. More precisely, we will work in a monster model $\cC \models T$, and $T^P$ will be the theory of $\cC^P$. In addition, assume that $P$ is embedded and stably embedded; i.e., there is no additional structure definable on $P$, which is not definable already in $T^P$ with parameters from $P$.  

One can ask the following natural question: does every model of $T^P$ ``occur'' as the $P$-part of some model of $T$? We know that, in general, the answer to this question is ``no''. So the question becomes: under what conditions on $T$ is the above the case for all models of $T^P$? Alternatively, one can ask: under what conditions on the model of $T^P$ is the above the case for an arbitrary $T$? These and related questions are addressed in several  papers, e.g., \cite{Sh234,ShUs322a, ShUs322b, Us24}. For example, in a recent paper \cite{Us24}, the author has showed that if $T$ is fully stable over $P$ (see Definition \ref{dfn:fullystable}), then the answer is ``yes'' for any model of $T^P$. In fact, more is true: any complete set $A$ (Definition \ref{dfn:complete}) can be extended to a model of $T$ without changing its $P$-part (completeness is an obvious necessary condition on $A$ for this to be possible). A similar result has been recently proven in a preprint by Bayes, Kaplan, and Simon assuming a different notion of stability  over $P$ (in particular, it holds for any simple theory $T$). 

One motivating example for this study is the theory of difference fields. Afshordel \cite{MR3343525} has shown that any difference field $(K, \sigma)$ with a pseudo-finite fixed field $k = Fix(\sigma)(K)$, can be extended to a model $M$ of $ACFA$ without changing the fixed field, i.e., $k = Fix(\sigma)(M)$. The author's result mentioned above generalizes this for the case of characteristic $0$: it is shown in \cite{Us24} that $ACFA_0$ is indeed fully stable over $P = Fix(\sigma)$. The results of Bayes, Kaplan, and Simon, generalizes this for all characteristics, since $ACFA_p$ is simple for any value of $p$.  At the same time, the results in \cite{Us24} give more: specifically, if $T$ is fully stable, and $A$ is a complete set, the model $M \models T$ extending $A$ with $P^A = P^M$ can be chosen to be \emph{locally constructible} over $A$. In particular, in case of $ACFA_0$, this shows that every difference field $k$ of characteristic $0$ with a pseudo-finite fixed field can be extended to a model $K$ of $ACFA_0$ with $Fix(\sigma)(K) = Fix(\sigma)(k)$ such that $K$ is locally constructible over $k$. It would be interesting to ask whether the same is true for $ACFA_p$. 

\medskip

This paper addresses a different (but related) question. Suppose $N \models T^P$ is $\lam$-saturated. Is there a $\lam$-saturated model $M$ of $T$ with $P^M = N$? If the answer is ``yes'', we say that $T$ has the $\lam$-existence property over $N$. One can ask for more, e.g., $M$ is $\lam$-prime, or even $\lam$-primary over $N$. 

If e.g. $\lam = \lam^{<\lam}$ and $N$ is a saturated model of $T^P$ of cardinality $\lam$, then the answer to the above questions in ``yes'' \cite{ShUs322a}. If, however, $N$ is a $\lam$-saturated model of $T^P$ of a larger cardinality, things become more complicated. 

The original motivation comes, again, from the study of difference fields. In  \cite{Chatzidakis2020RemarksAT}, Chatzidakis has shown that $ACFA_0$ has $\lam$-existence (over the fixed field) for all uncountable $\lam$ (and even $\lam = \aleph_\eps$). Moreover, every $\lam$-saturated pseudo-finite field $k$ of characteristic $0$ can be extended to a $\lam$-saturated model $K$ of $ACFA_0$ with $Fix(\sigma)(K) = k$ such that $K$ is $\lam$-prime over $k$. 

\smallskip

Motivated by Chatizidakis' work,  the author has shown in \cite{usvyatsov2024stability} that, given a theory $T$ fully stable over $P$ and $\lam>|T|$, the class $\mathcal{K}_\lam(T,N)$ of $\lam$-saturated models of $T$ with a fixed $P$-part $N$, admits prime models. This, however, relies on the class $\mathcal{K}_\lam(T,N)$ being non-empty. It was left open whether, given a prescribed model of $T^P$ which is $\lam$-saturated, it always ``occurs" as the $P$-part of some $\lam$-saturated model of $T$. In the case of $ACFA_0$, for example, this is not hard, and is one of the first steps in the proof of the main theorem in \cite{Chatzidakis2020RemarksAT}. In the general context of model theory over $P$, it is often the existence problems that prove more difficult, and require more sophisticated techniques (in our case, stability theory).  

In this paper we address this question and prove that if $T$ is fully stable over $P$, then $T$ has $\lam$-existence over $P$. Moreover, every complete set $A$ whose $P$-part is a $\lam$-saturated model $N$ of $T^P$ can be extended to a $\lam$-saturated model $M$ of $T$ with $P^M = N$ such that $M$ is $\lam$-constructible over $N$. This generalizes the results in \cite{Chatzidakis2020RemarksAT}, since $ACFA_0$ is fully stable.

It is interesting to point out that, although we do not know whether a similar result holds for $ACFA_p$ for $p>0$, Chatzidakis has conjectured that the answer is ``no''. This would immediately imply that $ACFA_p$ is not fully stable, even though it has full existence. It may also turn out to be that case that $ACFA_p$ has full existence, but the models of $ACFA_p$ that are constructed in \cite{MR3343525} are not locally atomic over the underline difference fields. These would all be interesting questions to investigate.

\section{Preliminaries}

In this section we set up the framework and recall basic notions and facts. The reader is referred to  previous works (e.g., \cite{ShUs322a, Us24}) for a more detailed discussion.

\subsection{Conventions and hypotheses}

\begin{conv}
Let $T$ be a complete first order theory, $P$ a monadic predicate in
its vocabulary. For simplicity, we will assume that the language of $T$ does not contain function symbols (so 
that every subset of a model, containing all the constant symbols, is a substructure; in fact, we will treat all subsets as substructures). In addition, we assume that $T$ implies that $P$ is infinite. 
\end{conv}

\noindent
Let $\mathcal{C} $ be the monster model of $T$. From now on, we assume that all models of $T$ are elementary submodels of $\mathcal{C}$, and all sets are subsets of $\mathcal{C}$. \medskip
 
For $M \models T$, we denote by  $M|_P$ the set $P^M$ viewed as a substructure of $M$. Similarly, for a subset $A \subseteq M$, we denote by $M|_A$ the substructure of $M$ with universe $A$. We write $A\equiv B$ if $Th({\cal C}|_A)=Th({\cal
C}|_B)$.

We also denote $\cC^P = \cC|_P$ and $T^P = Th(\cC|_P)$. For a set $A$, we denote $P^A = A \cap P^\cC$.

When no confusion should arise, we will write $P$ for $P^\cC$. Also, for a set $A$, we will often denote by $A$ both the set and the substructure of $\cC$ with universe $A$. So for example, when we write that $A\cap P^\cC$ is $\lam$-saturated, or just that $A\cap P$ is $\lam$-saturated, we mean that the substructure $\cC|_{A\cap P^\cC}$ is a $\lam$-saturated model of the appropriate theory (if $A \cap P \prec P$, which will be the case in this paper, then the appropriate theory is $T^P$).


%

\bigskip

Throughout the paper, we are going to make the following fundamental assumptions on $T$:

\begin{hyp}\label{asm:1}

  
\begin{enumerate}
 \item Every type over $P^\mathcal{C}$ is definable. In other words, $P$ is \emph{stably embedded}: subsets of $P^\mathcal{C}$  that are definable  with parameters (in $\mathcal{C}$), are already definable (in $\mathcal{C}$) with parameters in $P^\mathcal{C}$.
\item  In addition, subsets of $P^\mathcal{C}$  that are 0-definable (in $\mathcal{C}$),  are already 0-definable in
$\cC^P$ (modulo $T^P$).
\item Moreover, $T$ has quantifier elimination (even down to the level of predicates). 
 \end{enumerate}
\end{hyp}

%
%

See \cite{ShUs322a} for clarifications and ``justifications'' for making these assumptions.

\smallskip

Note that $\cC^P$ can be seen as the monster model for the theory $T^P$, and that it follows from Hypothesis 1(iii) that $T^P$ also has QE.

\subsection{The existence property, completeness, and stability}

The main concept under discussion in this paper is the $\lam$-existence property (see also Definition \ref{de:fullexist} below). We recall some basic concepts from \cite{Us24}, \cite{ShUs322a}, and \cite{ShUs322b}:

\begin{de}\label{de:existence}
\begin{enumerate}
\item 
	We say that a set $A$ has the \emph{existence property over $P$}, or simply the \emph{existence property} if there exists $M \models T$ such that $A \subseteq M$ and $P^M = P^A$. 
\item
	Let $\lam$ be a cardinal. If, in addition, in clause (i) there is $M$ which is $\lam$-saturated, we say that $A$ has the \emph{$\lam$-existence property over $P$}, or simply the \emph{$\lam$-existence property}. 
\end{enumerate}
\end{de}

Given a set $A$, in order for their to be a chance for $M$ as above to exist, $P^A$ should be ``suitable'' for being the $P$-part of a model of $T$. For example, $P^A$ should obviously be itself a model of $T^P$. In fact, this is not enough. $P^A$ should also be closed under the parameters necessary to define types of tuples of $A$ over $P$. These conditions are summarized in the following definition (see also Fact \ref{obs:complete_characterization}).


\begin{de}\label{6}\label{dfn:complete}
$A\subseteq {\cal C}$ is {\it complete} if for every
formula $\psi(\bar x,\bar y)$ and $\bar b\subseteq A, \models
(\exists\bar x\in P)\psi(\bar x,\bar b)$ implies $(\exists
\bar a\subseteq P\cap A)\models \psi(\bar a,\bar b)$. 
\end{de}

The following is clear:

\begin{obs}\label{obs:complete}
If $M\prec {\cal C}$ and $P^M\subseteq A\subseteq M$, then
$A$ is complete.
\end{obs}




\medskip

Clearly, in order for $A$ to have the $\lam$-existence property, $P^A$ also has to be $\lam$-saturated. 

\medskip

The following useful easy characterization offers another understanding of the notion of completeness (Observation 4.2 in \cite{ShUs322a}):

\begin{fact}
\label{obs:complete_characterization}
A set $A$ is complete if and only if for every $\bar a\subseteq A$ and
$\psi(\bar x,\bar y)$  the $\psi$-type $tp_\psi (\bar a/P^{\cal C})$ is definable
over $A\cap P^{\cal C}$ and $A\cap P^{\cal C}\prec P^{\cal C}$.
\end{fact}


%
%

\smallskip

Now that we have defined complete sets, we can recall the notion of full existence:

\begin{de}\label{de:fullexist}
 \begin{enumerate}
 \item We say that a theory $T$ has the \emph{full existence property}, or just has \emph{full existence}, if every complete set $A$ has the existence property.  
\item We say that a theory $T$ has the \emph{full $\lam$-existence property}, or just has \emph{full $\lam$-existence}, if every complete set $A$ with a $\lam$-saturated $P$-part has the $\lam$-existence property. 
\item Given $N \models T^P$, we say that $T$ has  \emph{full existence over $N$} if every complete set $A$ with $P^A = N$ has the existence property. 
\item Given $N \models T^P$, we say that $T$ has  \emph{full $\lam$-existence over $N$} if every complete set $A$ with $P^A = N$ has the $\lam$-existence property. 
\end{enumerate}

\end{de}

Let us recall one known general example of the $\lam$-existence property:

\begin{fact}(\cite[4.13]{ShUs322a})\label{fact:satexist}
Let $N$ be a saturated model of $T^P$ of cardinality $\lam = \lam^{<\lam}$. Then every complete $A$ with $|A| = \lam$ and $P^A = N$ has the $\lam$-existence property.
\end{fact}

\subsection{Relevant types and stability over $P$}

Let us now recall the relevant notion of a type for this context. Note that it is only defined over a complete set.

\begin{de}\label{6.5}\label{dfn:startypes}

Let $A$ be a complete set. 

\begin{itemize}

\item[(i)] Let
$$S_*(A)=\{tp(\bar c/A):P\cap (A\cup \bar c)=P\cap A {\rm\ and}\ A\cup \bar
c\ {\rm is\ complete}\}$$
\item[(ii)] $A$ is \emph{stable over $P$}, or simply \emph{stable}, if for all $A^\prime$
with $A^\prime\equiv A$, we have $|S_*(A^\prime)|\leq |A^\prime|^{|T|}$.
\end{itemize}
\end{de}

%
%

We often refer to types in  $S_*(A)$ as \emph{complete types over $A$ which are weakly orthogonal to $P$}.

\subsection{Some basic properties}

\begin{obs}\label{8.5}
 Let $A$ be a set and $\c$ a tuple. Then $A$ is complete and $tp(\c/A) \in S_*(A)$ if and only if for every formula $\psi(\x,\b,\z)$ over $A$ we have 
 \[ \models \exists \z \in P\, \psi(\c,\b,\z) \Longrightarrow \exists \bar d \in P\cap  A \, \text{such that} \, \models \psi(\bar c,\b,\d) \]
 \end{obs}

\medskip

Compactness implies that completeness can be characterized using a stronger uniform definability condition (see section 4 of \cite{Us24}):

\begin{fact}\label{obs:uniformdef}
In Hypothesis \ref{asm:1} (i), the definitions of $\psi$-types over $P$ can be assumed to be \emph{uniform}. Specifically:

 There are $\langle \Psi_\psi(\y,\z) :\psi(\bar
x,\bar y)\in L(T)\rangle $ such that for all
$\bar a\subseteq \cC$, $tp_\psi (\bar a/P^\cC)$ is definable by
$\Psi _\psi (\bar y,\bar c)$ for some $\bar c\subseteq P^\cC$. 

In other words, for every $\psi(\x,\y)$ and $\a \in \fC$ there exists $\c \in P^\cC$ such that $\cC \models \forall \y \psi(\a,\y) \longleftrightarrow \Psi(\y,\c)$.

\end{fact}
%
%
%
%

\begin{co} \label{co:defcomplete0}For any complete $A$ and
for all
$\bar a\subseteq A$, $tp_\psi (\bar a/P\cap A)$ is definable by
$\Psi _\psi (\bar y,\bar c)$ for some $\bar c\subseteq A\cap P$
(where $\Psi _\psi (\bar y,\bar z)$ is as in Observation \ref{obs:uniformdef}).
\end{co}

\bigskip

As an easy (but important)  consequence of Hypothesis \ref{asm:1}, we obtain that types over complete sets of elements ``in $P$'' are, 
in fact, types over $P^A$. Specifically (see Corollary 4.7 in \cite{ShUs322a}):

\begin{fact}\label{fct:typeoverP}
	Let $A$ be a complete set, $p(x)$ a (partial) type over $A$ with $P(x) \subseteq p$. Then $p$ is equivalent 
	to a $T^P$-type $p'$ over $P^A$ with $|p'|=|p|$. 
	
	In particular, if $p$ is finite, then it realized in $P^A$. Similarly, if $|p|<\lam$ and $P^A$ is $\lam$-compact.
\end{fact}

\smallskip

It is also useful to note that, since $T$ has QE, the property of completeness for a set $A$  depends only on its first order theory (as a substructure of $\cC$):

\begin{fact}(Lemma 4.5 in \cite{ShUs322a})\label{lem:complete_QE}\label{7.5}
\begin{itemize}
\item[(i)] If $A_1\equiv A_2$, then $A_1$ is complete iff $A_2$ is
complete.
\item[(ii)] $A$ is complete iff whenever the
sentence $$\theta=:(\forall \bar y)[S(\bar y)\longleftrightarrow (\exists x\in P)
R(x,\bar y)]$$ for quantifier free 
$R,S$ is satisfied in ${\cal C}$, then
$A$ satisfies $\theta$.
%
%
\end{itemize}
\end{fact}

\subsection{Stability and rank}

Finally, we recall a notion of rank that ``captures'' stability over $P$ (\cite{PiSh130}, \cite{ShUs322a}). Since we will explicitly use it in later arguments, we have decided to include the definition as a part of the preliminaries. 

\smallskip

\begin{de}\label{R}
For a complete set $A$, a (partial) $n$-type $p(\bar x)$ (with parameters in ${\cal C}$), 
  sets $\Delta _1,\Delta _2$ of formulas $\psi (\bar
x,\bar y)$, and a cardinal $\lambda$, we define when $R^n_A(p,\Delta
_1,\Delta _2,\lambda)\geq \alpha$. We usually omit $n$.

\begin{itemize}
\item[(i)] $R_A(p,\Delta _1,\Delta _2,\lambda)\geq 0$ if $p(\bar x)$ is consistent. 

\item[(ii)] For $\alpha$ a limit ordinal: $R_A(p,\Delta _1,\Delta
_2,\lambda)\geq \alpha$ if $R_A(p,\Delta _1,\Delta
_2,\lambda)\geq \beta$ for every $\beta <\alpha$.

\item[(iii)] For $\alpha =\beta +1$ and $\beta$ even:
For $\mu<\lambda$ and finite $q(\bar x)\subseteq p(\bar x)$ we can
find $r_i(\bar x)$ for $i\leq\mu$ such that;

\begin{itemize}
\item[{1.}] Each $r_i$ is a $\Delta _1$-type over $A$,

\item[{2.}] For $i\not= j, r_i$ and $r_j$ are explicitly contradictory
(i.e. for some $\psi$ and $\bar c$, $\psi (\bar x,\bar c)\in r_i,
\neg\psi(\bar x,\bar c)\in r_j$).

\item[{3.}] $R_A(q(\bar x)\cup r_i(\bar x), \Delta _1,\Delta
_2,\lambda)\geq \beta$ {for all $i$}.
\end{itemize}

\item[(iv)] For $\alpha =\beta +1: \beta$ odd: For 
$\mu<\lambda$ and finite $q(\bar x)\subseteq p(\bar x)$ and $\psi
_i\in \Delta _2, \bar b_i\in A$ ($i\leq \mu$), there are $\bar d_i\in
A\cap P$ such that $R(r_i,\Delta _1,\Delta
_2,\lambda)\geq \beta$ where $r_i=q(\bar x)\cup \{(\forall\bar
z\subseteq P) \left[\psi _i(\bar x,\bar b_i,\bar z)\equiv\Psi_{\psi
_i}(\bar z,\bar d_i)\right] \colon i<\mu\}$, and $\Psi_{\psi _i}$ is as in Corollary \ref{co:defcomplete0}.
\end{itemize}

$R^n_A(p,\Delta _1,\Delta _2,\lambda)= \alpha$ if $R^n_A(p,\Delta
_1,\Delta _2,\lambda)\geq \alpha$ but not $R^n_A(p,\Delta
_1,\Delta _2,\lambda)\geq \alpha +1$. $R^n_A(p,\Delta
_1,\Delta _2,\lambda)=\infty$ iff $R^n_A(p,\Delta
_1,\Delta _2,\lambda)\geq \alpha$ for all $\alpha$.
\end{de}

The main case for applications will be $\lambda =2$. Note that the
larger $R^n_A(p,\Delta _1,\Delta _2,\lambda)$, the more evidence there
is for the existence of many types $q(\bar x)\in S_*(A)$ consistent with
$p(\bar x)$.

\medskip 

See section 5 of \cite{ShUs322a} for a detailed discussion of some basic properties of the rank; we recall here only a few. 
\medskip

\begin{fact}\label{rank}(Fact 5.3 in \cite{ShUs322a})
%
For every $p$ there is a finite $q\subseteq p$, such that 
$R_A(p,\Delta _1,\Delta _2, 2)=R_A(q,\Delta _1,\Delta _2, 2)$.
%
%
\end{fact}

\begin{fact}\label{rankeven}(see Fact 5.3 in \cite{ShUs322a}). 
	Let $A$ be complete, $p\in S_*(A)$, $q^* \subseteq p$, and assume
	 $$R^n_A(q^*,\Delta _1,\Delta _2,\lambda)=R^n_A(p,\Delta _1,\Delta _2,\lambda)=k < \infty$$ Then $k$ is even.
\end{fact}

The following fact connects rank and stability. 

\begin{theorem}\label{13}\label{thm:stablerank}(\cite{PiSh130}, see also Theorem 5.4 in \cite{ShUs322a})

The following are equivalent:
\begin{itemize}
\item[(i)] $A$ is stable.
\item[(ii)] For every finite $\Delta _1$ and finite $n$ there are some
finite $\Delta _2$ and finite $m$ such that 
$R^n_A(\bar x=\bar x,\Delta _1,\Delta _2,2)\leq m$.
\end{itemize}
\end{theorem}

\begin{fact} (Corollary 5.5 in \cite{ShUs322a})\label{fct:satstable}
	In Definition \ref{6}(iv), it is not necessary to consider all $A' \equiv A$. 
	Specifically, a complete set $A$ is stable if and only if $|S_*(A')| \le |A'|^{|T|}$ for some $A' \equiv A$ saturated, $|A'|>|T|$. 
	
	Moreover, it is enough that for $A'$ as above, $|S_*(A')| < 2^{|A'|}$.
\end{fact}

\medskip

We  often omit the superscript and the subscript in the rank $R^n_A$, and write simply $R$ (at least when 
$n$ and $A$ are easily deduced from the context).

%

\section{True minimality of types}


Here we expand on some observations made in section 5 of \cite{Us24}. 
Our goal is to formulate general statements that will be convenient for later use (without further recollection of  the definition or explicit properties of the rank). The following lemma follows easily from the definitions of the rank, and the fact that it is finitary (Fact \ref{rank}):

\begin{lem}
\label{lem:minimal}
	Let $A$ be complete,
$p$ a 
partial type over $A$, and let $X$ be a property of partial types over $A$ such that $X$ holds for $p$. Assume, in addition, that $R^n_A(p,\Delta _1,\Delta _2, 2)=k < \omega$. 

Then there exists $q$ such that:

\begin{enumerate}
 \item
 	$q$ is a finite set of formulae over $A$
 \item
 	$p' = p \cup q$ is a partial type over $A$
\item 
	$X$ holds for $p'$
\item
	$R^n_A(p',\Delta _1,\Delta _2,\lambda)$ is minimal with respect to the  clauses (b) and (c). 
\end{enumerate}

%

\end{lem}

Note that the property $X$ in the previous definition does not have to be first order definable (and indeed usually it will not be). However, we will always implicitly assume that  $X$ \emph{implies that $p$ is indeed a type, i.e., is consistent}.

We call a type $p'$ satisfying the clauses (i) -- (iv) of the Lemma above \emph{a minimal extension of $p$ over $A$ with respect to $\Delta_1$, $\Delta_2$, $\lam$, and $X$}. If $p$ is clear from the context, or $p = [x=x]$, we say that $p'$ is a \emph{minimal type} with respect to $\Delta_1$, $\Delta_2$, $\lam$, and $X$.

In order to ensure minimal types have the minimality properties that one would expect, we need to strengthen the notion as follows:

\begin{de}
\begin{enumerate}
\item 
 We call a minimal extension $p'$ of $p$ over a set $A$ (with respect to $\Delta_1$, $\Delta_2$, $\lam$, and $X$), a \emph{truly minimal} extension if, in addition to the properties (i) -- (iv), it is also the case that the rank $R^n_A(p',\Delta _1,\Delta _2,\lambda)$ is even. We call $p'$ a truly minimal extension over $A$ with respect to $\Delta_1$,  $\lam$, and $X$ if it is truly minimal with respect to  $\Delta_1$, $\Delta_2$, $\lam$, and $X$ for some $\Delta_2$. We call it truly minimal with respect to $\Delta_1$ and $X$ if $\lam = 2$. Finally, we call $p'$ a \emph{\ph-truly minimal extension of $p$  over $A$ with respect to $X$} if it is truly minimal over $A$ with respect to $\Delta_1 = \set{\ph, \neg\ph}$ and $X$.
\item
We call $p'$ a \emph{truly minimal extension} of $p$ over $A$ with respect to $X$ if it is $\ph$-truly minimal over $A$ with respect to $X$ for every formula $\ph$. 
\end{enumerate}
\end{de}


\begin{obs}\label{obs:phiminimal}
 Let $p$ is a {\ph-truly minimal extension} of a type $p_0$ with respect to a property $X$ over a set $A$. Then for no $a \in A$ do both $p \cup\set{\ph(x,a)}$ and $p \cup \set{\neg\ph(x,a)}$ have the property $X$.  
\end{obs}
\begin{proof}
Let $p^+ = p \cup\set{\ph(x,a)}$, $p^- = p\cup \set{\neg\ph(x,a)}$, and assume that both $p^+$ and $p^-$ have the property $X$. 

Let $\Delta_1 = \set{\ph,\neg\ph}$, and let $\Delta_2$ be as in the definition of $\ph$-true minimality. In particular, $R(p,\Delta_1,\Delta_2,2) = k$ is even. By Fact \ref{rank}, there is $q \subseteq p$  finite such that $k=R(q,\Delta_1,\Delta_2,2)$.  Denote $q^+ = q \cup\set{\ph(x,a)}$, $q^- = q\cup \set{\neg\ph(x,a)}$.

Since $p$ is a minimal extension with respect to ($\Delta_1, \Delta_2, 2$) and $X$, $R(p,\Delta_1, \Delta_2, 2) \le R(p^+ \Delta_1, \Delta_2, 2)$. Since $p \subseteq p^+$, these ranks are, in fact, equal; similarly for $p^-$. So $R(p^+,\Delta_1,\Delta_2,2) = R(p^-,\Delta_1,\Delta_2,2) = k$.

Now, $k = R(q,\Delta_1,\Delta_2,2) \ge R(q^+,\Delta_1,\Delta_2,2) \ge R(p^+, \Delta_1, \Delta_2, 2) = k$; similarly for $q^-$. So we have:
\begin{enumerate}
\item 
	$q \subseteq p$ finite.
\item
	$q^+,q^-$ are explicitly $\Delta_1$-contradictory extensions of $q$ of rank $k$.
\item
	$k$ is even.
\end{enumerate}

Therefore, by the definition of the rank, $R(p,\Delta_1,\Delta_2,2) \ge k+1$, a contradiction.
\end{proof}

\medskip

Lemma \ref{lem:minimal} tells us that any type of finite rank can be extended to a minimal type; we shall revisit it later (see Lemma \ref{lem:rankeven0}) to show that under certain conditions it can be strengthened to obtain a truly minimal type.

\section{Complete sets with saturated $P$-part}

\begin{de}
 Let $A$ be a set, $\lam$ an infinite cardinal. We say that  $A$ is \emph{$\lam$-complete } if for every (partial) type $p$ over $A$ of cardinality $<\lam$, if $p$ is realized in $P^\cC$, then it is realized in $P^A$. 
\end{de}

Clearly, $A$ is $\aleph_0$-complete if and only if it is complete, and if $A$ is $\lam$-complete for some $\lam$, then it is also $\mu$-complete for all $\mu < \lam$, in particular, $\lam$-completeness implies completeness. 

\begin{lem}
 A set $A$ is $\lam$-complete if and only if $A$ is complete and $P^A$ is $\lam$-compact. 
\end{lem}
\begin{proof}
 The ``only if'' direction is clear. The ``only if'' direction follows from Fact \ref{fct:typeoverP}.
\end{proof}

\smallskip
\begin{theorem}\label{co:complete_char}
Let $A$ be a set. Then the following are equivalent:
\begin{enumerate}
\item $A$ is $\lam$-complete.
\item For every partial  type $p(\x)$ over $A$ of size $<\lam$ (not necessarily realized in $A$) and every  collection \set{\psi_i(\x,\b_i,\z_i) \colon i<\ka} of formulae over $A$ (with $\ka<\lam$), there are  $\set{\d_i \colon i<\ka} \subseteq P^A$ such that 
the following is a type over $A$:

\[
	p(\x) \cup \left\{  \left[(\exists \z_i \subseteq P)  \psi_i(\x,\b_i,\z_i)\right] \longrightarrow \psi_i(\x,\b_i,\d_i)  \colon 
	i<\ka \right\}
\]

%
%
\item
$P^A \elem P^\cC$ is $\lam$-compact, and:

For every partial  type $p(\x)$ over $A$ of size $<\lam$ (not necessarily realized in $A$) and every 
 collection \set{\psi_i(\x,\b_i,\y_i) \colon i<\ka} of formulae over $A$ (with $\ka < \lam)$, there are  $\set{\d_i \colon i<\ka} \subseteq P^A$ such that  the following is a  type over $A$:
\[
p(\x) \cup \left\{  (\forall \y_i \subseteq P) \left[ \psi_i(\x,\b_i,\y_i) \longleftrightarrow \Psi_{\psi_i}(\y_i,\d_i) \right]  \colon 
	i<\ka \right\}
%
\]
where  $\Psi_{\psi_i}(\y_i,\z_i)$ is the defining formula for $\psi_i = \psi_i(\x\x^i, \y_i)$ as in Corollary \ref{co:defcomplete0}; so $\len(\x^i) = \len(\b^i)$.

\item
$P^A \elem P^\cC$ is $\lam$- compact, and:

For every partial  type $p(\x)$ over $A$ of size $<\lam$ (not necessarily realized in $A$) and a formula $\psi(\x,\b,\y)$  over $A$ , there is  $\d \subseteq P^A$ such that  the following is a type over $A$:
\[
p(\x) \cup \left\{  (\forall \y \subseteq P) \left[ \psi(\x,\b,\y) \longleftrightarrow \Psi_{\psi}(\y,\d) \right]  \colon 
	i<k \right\}
%
\]
where  $\Psi_\psi(\y,\z)$ is the defining formula for $\psi = \psi(\x\x',\y)$.

\end{enumerate}

\begin{proof}
 (i) $\implies$ (ii): Assume $A$ is $\lam$-complete, and let $p$, $\psi_i$ be as in (ii). For every finite $p_0 \subseteq p$
 and every finite $W \subseteq \ka$, consider the following formula:
 
 \[
 	\Phi_{p_0, W}(\z_i \colon i \in W) = \exists \x \left(\bigwedge p_0(x) \bigwedge_{i\in W} \left[(\exists \z_i \subseteq P)  \psi_i(\x,\b_i,\z_i)\right] \longrightarrow \psi(\x,\b_i,\z_i)\right)
 \]
 
 And let 
 
 \[
 	\Delta  = \left\{  \Phi_{p_0}(\z_0, \ldots, \z_{k-1}) \colon p_0 \subseteq p, W \subseteq \ka \text{ \emph{finite}} \right\}
 \]

 Clearly, $\Delta$ is partial type over $A$ of size $\le |p| + \ka < \lam$ which is realized in $P^\cC$. Since $A$ is $\lam$-complete, there exist $\d = \seq{d_i \colon i<k} \subseteq P^A$ such that $\cC \models \Delta(\d)$; it is easy to see that they are as required in (ii).
 
%

The implications (iii) $\implies$ (iv) and (iv) $\implies$ (i) are clear  (for the latter, take e.g. $p_0 = [x=x]$). Similarly for (ii) $\implies$ (i). 
 
 
%

 (i) $\implies$ (iv): Since $A$ is complete, $P^A \elem P^\cC$.  Now let $p$, $\psi_i$ be as in (iv). Similarly to the previous proof, consider (for every finite $p_0 \subseteq p$ and $W \subseteq \ka$):
 
\[
 	\Phi = \Phi_{p_0, W}(\z_i \colon i \in W) = \exists \x \left(\bigwedge p_0(x) \bigwedge_{i \in W} (\forall \y_i \subseteq P) \left[ \psi_i(\x,\b_i,\y_i) \longleftrightarrow \Psi_{\psi_i}(\y_i,\z_i) \right] \right)
 \]
 
Again, we claim that $\cC\models (\exists \z \subseteq P) \Phi( \z)$.  Indeed, we are not asking (yet) that $\z$ be in $P^A$; so taking any $\a \models p$ (for $\x$), and any $\d'_i \in P^\cC$ as in Observation \ref{obs:uniformdef} for each $\psi_i(\a\b_i,\y_i)$ works. 

Now let $\Delta = \set{\Phi_{p_0,W} \colon p_0, W \text{ finite}}$. Since $\Delta$ is over $A$, and $A$ is $\lam$-complete, we can find $\d$ in  $P^A$ realizing $\Delta$, and it is as required in (iv).

\end{proof}

\end{theorem}

\begin{lem}\label{lem:rankeven0}
 	
	Let $A$ be $\lam$-complete,
	$p$ a  partial type over $A$ of cardinality $<\lam$, and that $R^n_A(p,\Delta _1,\Delta _2, \mu) < \omega$ for some $\Delta_1, \Delta_2, \mu$. 
	
	Let 
	$p'$ be a minimal extension of $p$ with respect to $\Delta_1$,   $\Delta_2$, $\lam$, and the property ``of cardinality $<\lam$''. Then $p'$ is truly minimal with respect to the above $\Delta_1$,   $\Delta_2$, $\mu$, and property. 


\end{lem}

\begin{proof}
Let $R^n_A(p',\Delta _1,\Delta _2,\mu)=k < \omega$. We need to show that $k$ is even. Assume that it is odd; we shall show that 
$R^n_A(p',\Delta _1,\Delta _2,\mu)\ge k+1$. 
	
\smallskip

Let $\psi
_i\in \Delta _2, \bar b_i\in A$ as in clause (iv) of the definition of the rank (Definition \ref{R}).  By Theorem \ref{co:complete_char}, since $A$ is $\lam$-complete,  
there are $\d_i \in P^A$ such that $p'(\x)$ is consistent with the set 
\[
	\pi(\x,\d) = 
	\left\{(\forall\bar z_i \subseteq P) \left[\psi _i(\bar x,\bar d_i,\bar z_i)\longleftrightarrow \Psi_{\psi_i}(\bar z_i,\bar c_i)\right] \colon i
  \right\}
\]

Let $p''(\x) = p'(\x) \cup \pi(\x,\d)$. It is still a type over $A$ of cardinality $<\lam$. 
Hence by minimality of $p'$, $R(p''(\x), \Delta_1, \Delta_2, \mu) = R(p'(\x), \Delta_1, \Delta_2, \mu) = k$. 

Now let $r_i=q(\bar x)\cup \{(\forall\bar
z\subseteq P) \left[\psi _i(\bar x,\bar b_i,\bar z)\equiv\Psi_{\psi
_i}(\bar z,\bar d_i)\right]$ for all  $i$ (again, as in clause (iv) of the definition of the rank). Note that $r_i \subseteq p''(\x)$, hence 
$R(r_i(\x), \Delta_1, \Delta_2, \mu) \ge R(p''(\x), \Delta_1, \Delta_2, \mu) \ge k$ for all $i$. 

So by (iv) of Definition \ref{R}, $R_A(p',\Delta _1,\Delta _2,\mu)\geq k+1$, concluding the proof.


\end{proof}
 
 Combining the above lemma with Lemma \ref{lem:minimal} and Fact \ref{thm:stablerank}, we conclude:

\begin{co}
\label{co:trulyminimal}
\begin{enumerate}
\item Let $A$ be $\lam$-complete,
$p$ a 
partial type over $A$ of cardinality $<\lam$, and let $X$ be a property of partial types over $A$ such that $X$ holds for $p$. Assume, in addition, that $R^n_A(p,\Delta _1,\Delta _2,\mu)=k < \omega$. Then there exists a partial type $p'$ over $A$ extending $p$, which is truly minimal with respect to $\Delta _1,\Delta _2,\mu$, and $X$.
\item Let $A$ be stable and $\lam$-complete, $p$ a 
partial type over $A$ of cardinality $<\lam$, and let $X$ be a property of partial types over $A$ such that $X$ holds for $p$. Then for every $\Delta_1$ there exists a partial type $p'$ over $A$ extending $p$, which is truly minimal with respect to $\Delta _1$ and $X$.

\end{enumerate}

\end{co}

In particular, we have:

\begin{co}
 \label{co:phimin}
 Let $A$ be stable and $\lam$-complete, $p$ a 
partial type over $A$ of cardinality $<\lam$, and let $X$ be a property of partial types over $A$ such that $X$ holds for $p$. Then for every formula $\ph$ there exists a partial type $p'$ over $A$ extending $p$, which is $\ph$-truly minimal with respect to $X$.
\end{co}

\section{Isolation and prime models}

Let us recall the notion of isolation relevant for the discussion in this article. 

\begin{de}


A (partial) type $p$ over a set $A$ is called \emph{$\lam$-isolated} if there exists a subset $r \subseteq p$ with $|r|<\lam$ such that $r \equiv p$. We say that $p$ is $\lam$-isolated over $B \subseteq A$ if $r$ as above is a partial type over $B$.

\end{de}

\smallskip

The following Lemma with play a crucial role in our constructions It states that a  $\lam$-isolated type over a $\lam$-complete set is \emph{always} weakly orthogonal to $P$.


\begin{lem}\label{lem:isolated_star}
 Let $B$ be a $\lam$-complete set, $p \in S(B)$  $\lam$-isolated. Then $p \in S_*(B)$. 
\end{lem}
\begin{proof}
Consider the formula $\theta(\x,\y,\bar u) = (\forall \z \subseteq P)\left[ \psi(\x,\y,\z) \longleftrightarrow \Psi_\psi(\z,\y,\bar u)\right]$. 

It is enough to show that for every formula $\psi(\x,\b,\z)$ over $A$ there is $\d \in P^A$ such that 
\[
\theta(\x,\b,\d) = (\forall \y \subseteq P) \left[ \psi(\x,\b,\y) \longleftrightarrow \Psi_\psi(\y,\d) \right]  \in p
\]
where  $\Psi_\psi(\y,\z)$ is as in Corollary \ref{co:defcomplete0}.

Since $p$ is locally isolated, there is a finite $p_0 \subseteq p$ such that $p_0 \vdash p \rest \theta$. 

By clause (iv) of Theorem \ref{co:complete_char}, 
 there is  $\d \subseteq P^A$ such that  the following is a type over $A$:
\[
\pi(\x) = p_0(\x) \cup \left\{  \theta(\x,\b,\d) 
\right\}
%
\]

In other words, there is a complete type $q$ over $B$ extending $\pi(\x)$. But $p_0$ implies a complete $\theta$-type; so $q\rest\theta = p\rest\theta$. In particular, $\theta(\x,\b,\d) \in p$, as required.

\end{proof}

\begin{de}\label{dfn:primary}
Let $N$ be a model, $P^N \subseteq B \subseteq N$.
\begin{enumerate}
\item
	We say that the sequence $\d = \{d_i:i<\alpha\} \subseteq N$ is a $\lam$-construction over $B$ in $N$ if for all $i<\al$, the type $tp\{d_i/B\cup\{d_j:j<i\})$ is $\lambda$-isolated.

\item
	We say that a set $C \subseteq N$ is $N$ is \emph{$\lam$-constructible} over $B$ in $N$ if 
	there is a $\lam$-construction $\d$ over $B$ in $N$. 
	
	In particular, we say that $N$ is
	\emph{$\lam$-constructible} over $B$ if 
there is a construction $N=B\cup\{d_i:i<\al\}$ such that for all $i<\lam$ the type 
$tp\{d_i/B\cup\{d_j:j<i\})$ is $\lambda$-isolated. 
\item 
	We say that a model $N$ is \emph{$\lambda$-primary} over $B$ if it is $\lam$-constructible and $\lam$-saturated. 
\end{enumerate}
\end{de}

\begin{de}\label{dfn:prime}
Let $N$ be a model, $P^N \subseteq B \subseteq N$.
%


	We say that $N$ is \emph{$\lam$-atomic} over $B$ if 
for every $\bar d\subseteq N$, $tp(\bar d,B)$ is
$\lambda$-isolated over some $B_{\bar d}\subseteq B, |B_{\bar
d}|<\lambda$.
%
%
\end{de}

\section{The $\lam$-existence property}

In this section we prove the full $\lam$-existence property for fully stable  theories. 


\begin{lem}\label{lem:li} 
Let $\lam > |T|$, $B$ be a $\lam$-complete stable set, $p(\bar x)$ 
a partial type over $B$, $|p(\bar x)|<\lambda$.
\begin{enumerate}
\item
Let $\psi(\x,\y)$ a formula. Then there is
a finite type $q(\bar x)$ over $B$ 
such that 
$p(\bar x)\cup q(\bar x)$ is consistent, and it implies a complete $\psi$-type over $B$.
\item
There is
$q(\bar x)$ such that $|q(\bar x)|\leq |T|$,
$p(\bar x)\cup q(\bar x)$ consistent and there is $r\in S_*(B)$ such
that $p(\bar x)\cup q(\bar x)\equiv r(\bar x)$. 

In particular,  $r(\bar x)$ is
$\lambda$-isolated.
\item
The previous clause is also true if $\x$ is an infinite tuple with $<\lam$ variables, but in this case we can only 
require that  $|q|<\lam$.

Specifically, if $|\x| = \ka < \lam$, then there exists $|q| \le |T|\cdot\ka$ as above.

\end{enumerate}

\end{lem}

\begin{proof} 

\begin{enumerate}
\item 
Let $\{\psi _i(\bar x,\bar y_i):i<|T|\}$ list all
formulas of $L(T)$. 
Define $q_i(\bar x)$ by induction
on $i < |T|$ such that
\begin{itemize}
\item[(a)] $q_i$ is finite and is over $B$,
\item[(b)] $p(\bar x)\cup \bigcup _{j\leq i}q_j(\bar x)$ is
(consistent and) $\psi_i$-truly minimal with respect to the property $X=$``consistent and of cardinality $<\lam$''.
\end{itemize}

This is possible by Corollary \ref{co:phimin}, since $B$ is stable. Denote $p_i = p(\bar x)\cup \bigcup _{j\leq i}q_j(\bar x)$. Let $\b \in B$. Clearly, both $p_i \cup \set{\psi_i(x,\b)}$ and $p_i \cup \set{\neg\psi_i(x,\b)}$ are of cardinality $<\lam$. Therefore by Observation \ref{obs:phiminimal}, only one of the two sets above is a partial type over $B$. In other words, $p_i$ implies a complete $\psi_i$-type over $B$. 

\item

The type $p_{|T|} = \bigcup_{i<|T|}p_i$, where $p_i$ are as in the proof of the previous clause, is clearly as required (note that $|T|<\lam$, hence each $p_i$ satisfies the hypothesis of clause (i)). Note that $p \in S_*(B)$ by Lemma \ref{lem:isolated_star}.

\item

Also easy. 

\end{enumerate}

\end{proof}

\smallskip


Let us recall the notion of full stability over $P$ from \cite{Us24}.

\begin{de}\label{dfn:fullystable}
We say that $T$ is \emph{fully stable} over $P$ if every complete set $A \subseteq \cC$ is stable over $P$. 
\end{de}

It may be of interest to mention some weaker (more local) notions. We will prove slightly more precise results using these definitions. 

\begin{de}\label{dfn:fullystable2}
\begin{enumerate}
\item 
Given a complete set $A$, we say that $A$ is fully stable over $P$ if every $B \subseteq A$ with $P^A = P^B$ is stable over $P$. 
\item Given $N \models T^P$, we say that $T$ is fully stable over $N$ if every complete set $A$ with $P^A = N$ is stable over $P$. 
\end{enumerate}

\end{de}

\begin{lem}\label{lem:} 
 	Assume that $\lam > |T|$, $A$ is a $\lam$-complete set, and $T$ is fully stable over $P^A$. Then there exists $B \supseteq A$ such that 
	\begin{enumerate}
	\item
		$|B| \le |A| ^{|T|}$
	\item
		$B$ is complete, $P^B = P^A$
	\item
		Every  type over $A$ of cardinality $<\lam$ is realized in $B$	
	\item
		$B$ is $\lam$-constructible over $A$
	\end{enumerate}

\end{lem}
\begin{proof}
 	Let $I^{\lam}_* = \lseq{p}{i}{\delta}$ list all the $\lam$-isolated types in $S_*(A)$. Since $A$ is stable, $\delta \le  |A| ^{|T|}$. By Lemma \ref{lem:li}(ii), every small (of cardinality $<\lam$) partial type over $A$ can be extended to an element of $I^{\lam}_*$. Therefore, it is enough to construct $B$ that satisfies clauses (i), (ii), and (iv) above, which, in addition, realizes all the types in $I^{\lam}_*$.
	
	Now construct an increasing continuous sequence of sets $\lseq{A}{i}{\delta}$ such that:
	\begin{enumerate}	
	\item 
		$A_0 = A$
	\item
		$A_{i+1} = A_i \cup \set{\b_i}$
	\item 
		$P^{A_i} = P^A$
	\item 
		$A_i$ is stable
	\item
		$\models p(\b_i)$
	\item
		$\tp(\b_i/A_i) \in S_*(A_i)$ and is  $\lam$-isolated
	\end{enumerate}
	
	For successor stages, note that $p_i$ is $\lam$-isolated by, say, a small type $\pi_i$ over $A$. Since $A_i$ is stable, we can apply Lemma \ref{lem:li}, now viewing it as a small type over $A_i$. This results in a complete type $\tp(\b_i/A_i) \in S_*(A_i)$ extending $\pi_i$, which is $\lam$-isolated. Since $\models \pi_i(\b_i)$, we also have $\models p_i(\b_i)$. The rest should be clear. 
	
	For limit stages, only clause (iv) above is a problem, and here we use full stability over $P^A$.  Note that $P^{A_i} = P^A$ for all $i$. 
	
	Clearly, $B = A_\delta$ is as required.
%
\end{proof}

\begin{re}\label{re:stab_unions}
 Note that since, in the proof above, every $\b_i$ is finite, stability of the set $A_{i+1}$ follows from that of $A_i$. Hence the assumption of stability over a certain $N \models T^P$ can be replaced by the assumption that e.g. stability over $P$ is closed under taking increasing unions. 
\end{re}

Denote by $reg(\lam)$ the minimal regular cardinal $\ge \lam$; i.e., $reg(\lam)$ is either $\lam$ (if it is regular) or $\lam^+$.

\begin{theorem}\label{th:primary}
	Assume $\lam > |T|$, $N \models T^P$ is $\lam$-compact, $T$ is fully stable over $N$. Then $T$ has the full $\lam$-existence property over $N$. 
	
	More specifically: Let $A$ be a complete set with $P^A = N$. Then there exists a $\lam$-saturated $M \models T$ containing $A$ with $P^A = P^M$ and  $|M|\le |A|^{|T|}+reg(\lam)$. 
	Moreover, $M$ in the previous clause is  $\lam$-constructible (hence $\lam$-primary) over $A$. 
%
%
%
\end{theorem}
\begin{proof}
	Construct by induction on $i<reg(\lam)$ sets $A_i$ such that:
	\begin{enumerate}	
	\item 
		$A_0 = A$
	\item
		$|A_i| \le |A|^{|T|}$
	\item
		Every type of cardinality $<\lam$ over $A_i$ is realized in $A_{i+1}$
	\item
		$P^{A_{i+1}} = P^{A_i}$
	\item 
		$A_{i+1}$ is  $\lam$-constructible over $A_i$. 
	\end{enumerate}
	This is possible by the previous lemma. Now for some $\delta \le reg(\lam)$, $M = A_\delta = \bigcup_{i<\delta}A_i$ is clearly a $\lam$-saturated model which is  $\lam$-constructible over $A$, $P^M = P^A$, as required.

\end{proof}

\begin{co}
Assume that $T$ is  fully stable over $P$. Then $T$ has the full $\lam$-existence property (Definition \ref{de:fullexist}(ii)). 
\end{co}

\bibliography{common.bib}
\bibliographystyle{alpha}

\Addresses

\end{document}